\documentclass[12pt]{amsart}
\usepackage[top=0.9in, bottom=1in, left=1in, right=1in]{geometry}
\usepackage{amsmath, amssymb, amsthm}
\usepackage{verbatim}
\usepackage{graphicx}
\usepackage{enumerate}
\usepackage{mathrsfs}

\newtheorem{thm}{Theorem}[section]
\newtheorem{prop}[thm]{Proposition}
\newtheorem{lem}[thm]{Lemma}
\newtheorem{cor}[thm]{Corollary}

\def\XXint#1#2#3{{\setbox0=\hbox{$#1{#2#3}{\int}$ }
\vcenter{\hbox{$#2#3$ }}\kern-.6\wd0}}

\theoremstyle{definition}
\newtheorem{definition}[thm]{Definition}
\newtheorem{example}[thm]{Example}

\theoremstyle{remark}

\newtheorem{remark}[thm]{Remark}

\numberwithin{equation}{section}

\newtheoremstyle{ser}% name
{8pt}% Space above
{8pt}% Space below
{\it}% Body font
{}% Indent amount
{\sf}% Theorem head font
{:}% Punctuation after theorem head
{6mm}% Space after theorem head
{}% Theorem head spec (can be left empty, meaning `normal')

\newtheoremstyle{serr}% name
{8pt}% Space above
{8pt}% Space below
{\normalfont}% Body font
{}% Indent amount
{\sf}% Theorem head font
{.}% Punctuation after theorem head
{6mm}% Space after theorem head
{}% Theorem head spec (can be left empty, meaning `normal')

\theoremstyle{ser}

\theoremstyle{serr}

\begin{document}

%%
%% The title of the paper goes here.  Edit to your title.
%%

\title[Extensions of the Hadamard Determinant Inequality]{The Hadamard Determinant Inequality - Extensions to Operators on a Hilbert Space}

%%
%% Now edit the following to give your name and address:
%% 

%%\author{Soumyashant Nayak}
%%\ead{nsoum@upenn.edu}
%%\ead[url]{www.math.upenn.edu/$\sim$nsoum}
%%\address{Smilow Center for Translational Research, University of Pennsylvania, Philadelphia, PA 19104}
\author{Soumyashant Nayak}
\email{nsoum@upenn.edu}
\urladdr{www.math.upenn.edu/$\sim$nsoum} % Delete if not wanted.
\address{Smilow Center for Translational Research, University of Pennsylvania, Philadelphia, PA 19104}
%%
%% If there is another author uncomment and edit the following.
%%

%\author{Second Author}
%\address{Department of Mathematics, University of South Carolina,
%Columbia, SC 29208}
%\email{second@math.sc.edu}
%\urladdr{www.math.sc.edu/$\sim$second}

%%
%% If there are three of more authors they are added in the obvious
%% way. 
%%

%%%
%%% The following is for the abstract.  The abstract is optional and
%%% if not used just delete, or comment out, the following.
%%%

\begin{abstract}
A generalization of classical determinant inequalities like Hadamard's inequality and Fischer's inequality is studied. For a version of the inequalities originally proved by Arveson for positive operators in von Neumann algebras with a tracial state, we give a different proof. We also improve and generalize to the setting of finite von Neumann algebras, some `Fischer-type' inequalities by Matic for determinants of perturbed positive-definite matrices. In the process, a conceptual framework is established for viewing these inequalities as manifestations of Jensen's inequality in conjunction with the theory of operator monotone and operator convex functions on $[0,\infty)$. We place emphasis on documenting necessary and sufficient conditions for equality to hold.
\end{abstract}

%%
%%  LaTeX will not make the title for the paper unless told to do so.
%%  This is done by uncommenting the following.
%%

% \maketitle

%%
%% LaTeX can automatically make a table of contents.  This is done by
%% uncommenting the following:
%%

%\tableofcontents

%%
%%  To enter text is easy.  Just type it.  A blank line starts a new
%%  paragraph. 
%%

\maketitle

%%%%%%%%%%%%%%%%%%%%%%%%%%%%%%%%%%%%%%%%%%%%%%%%%%%%%%%%%%%%%%%%%%%%%%
\section{Introduction}
%%%%%%%%%%%%%%%%%%%%%%%%%%%%%%%%%%%%%%%%%%%%%%%%%%%%%%%%%%%%%%%%%%%%%%

At their core, the many applications of determinants in mathematical analysis are based on the geometric interpretation of the determinant of a square matrix as the (signed) volume of an $n$-parallelepiped with sides as the column vectors of the matrix. For instance, the change-of-variables formula in multidimensional integration involves the determinant of the Jacobian matrix. The study of estimates for the determinant of a matrix in terms of determinants of its principal submatrices is often useful as information about compressions of a matrix $A$ to certain subspaces ({\it i.e.}\ $PAP$ for a projection $P$) may be more readily available. An element of $M_n(\mathbb{C})$, the set of complex $n \times n $ matrices, is said to be {\it positive-semidefinite} if it is Hermitian with non-negative eigenvalues, and {\it positive-definite} if it is positive-semidefinite with strictly positive eigenvalues. Let $A$ be a positive-definite matrix with $(i, j)^{\textrm{th}}$ entry denoted by $a_{ij}$. Hadamard's inequality (\cite{hadamard}) states that  the determinant of a positive-definite matrix is less than or equal to the product of the diagonal entries of the matrix {\it i.e.}\ $\det A \le \prod_{i=1}^n a_{ii}$. Further, equality holds if and only if $A$ is a diagonal matrix. As a corollary, which is usually referred to by the same name, we get that the absolute value of the determinant of a square matrix is less than or equal to the product of the Euclidean norm of its column vectors (or alternatively, row vectors). In the case of real matrices, the inequality conveys the geometrically intuitive idea that an $n$-parallelepiped with prescribed lengths of sides has largest volume if and only if the sides are mutually orthogonal. An important application of this inequality to the theory of integral equations is in proving convergence results in classical Fredholm theory (\cite[section 5.3]{int-eqn}). More generally, a similar inequality, known as Fischer's inequality (\cite{fischer}), holds if one considers the principal diagonal blocks of a positive-definite matrix in block form. Hadamard's inequality is a corollary of Fischer's inequality by considering blocks of size $1 \times 1$.

For $n \in \mathbb{N}$, we denote the indexing set $\{ 1, 2, \cdots, n \}$ by $\langle n \rangle$. In Fischer's inequality below, for an $n \times n$ matrix $A$ and $\alpha \subseteq \langle n \rangle$, the principal submatrix of $A$ from rows and columns indexed by $\alpha$ is denoted by $A[\alpha]$. 

\begin{thm}[Fischer's inequality]
\label{thm:det-ineq-fischer}
{\textsl Let $A$ be a positive-definite matrix in $M_n(\mathbb{C})$. Let $\alpha_i \subseteq \langle n \rangle$ for $i \in \langle k \rangle$ such that $\alpha_i \cap \alpha_j  = \varnothing$ for $i, j \in \langle k \rangle, i \ne j$. Then $$\det(A[\cup_{i=1}^k \alpha_i ]) \le \prod_{i=1}^k \det(A[\alpha_i])$$ with equality if and only if $A[\cup_{i=1}^k \alpha_i] = P \mathrm{diag}(A[\alpha_1], \cdots, A[\alpha_k]) P^{-1}$ for some permutation matrix $P$.}
\end{thm}
Unwrapping the notation in the above theorem, we have that the determinant of a positive-definite matrix (in block form) is less than or equal to the product of the determinants of its principal diagonal blocks, with equality if and only if the entries outside the principal diagonal blocks are all $0$. We state an application of this result to information theory. For a multivariate normal random variable $(X_1, X_2, \cdots, X_n)$ with mean $\mathbf{0}$, covariance matrix $\Sigma$ and hence density $$f(\mathbf{x}) = \frac{1}{(2\pi)^{n/2} \det(\Sigma)^{1/2}}\exp(-\frac{1}{2}\mathbf{x}^T \Sigma ^{-1} \mathbf{x}), \mathbf{x} \in \mathbb{R}^n,$$ the Shannon entropy is  given by, $$h(X_1, \cdots, X_n) = - \int_{\mathbb{R}^n} f \log f \; d\mathbf{x} = \frac{1}{2} \ln ((2\pi e)^n \det(\Sigma)).$$ Fischer's inequality conveys the sub-additivity of entropy in the case of normal random variables \textit{i.e.}\ if the collection of normal random variables $X_1, \cdots, X_n$ is partitioned into disjoint subcollections, the sum of the entropies of the subcollections is bigger than the entropy of the whole collection, with equality if and only if the subcollections are mutually independent.

In \cite[Corollary 4.3.4]{general-hadamard}, Arveson obtains a generalized version of Hadamard's inequality for von Neumann algebras with a tracial state $\tau$, involving the Fuglede-Kadison determinant denoted by $\Delta$, which we paraphrase below. 

\begin{thm}
\label{thm:arveson-hadamard}
Let $\Phi$ be a $\tau$-preserving conditional expectation on a von Neumann subalgebra $\mathscr{S}$ of a von Neumann algebra $\mathscr{R}$. Then $$\Delta(A) \le \Delta(\Phi(A)),$$ for every positive $A$ in $\mathscr{R}$.
\end{thm} 

In this article, Theorem \ref{thm:general-hadamard} along with Remark \ref{rmrk:general-hadamard} gives us a proof of the above theorem which is different from the one in \cite{general-hadamard}. For the subsequent comments, we make the additional assumptions of faithfulness of the tracial state $\tau$ and regularity of the positive operator $A$ unless stated otherwise. Note that the first assumption necessitates the finiteness of the von Neumann algebra $\mathscr{R}$. In this setting, the new proof has the added advantage of directly yielding the conditions under which equality holds, given by $\Delta(A) = \Delta(\Phi(A)) \Leftrightarrow \Phi(A) = A$. In Theorem \ref{thm:ineq-square-inverse}, using this equality condition we are able to prove that $\Phi(A^{-1}) = \Phi(A)^{-1}$ if and only if $\Phi(A) = A$ which is somewhat surprising as the statement itself has no direct reference to $\Delta$. Our investigation reveals that this offers a small glimpse of a bigger picture. For instance, in Theorem \ref{cor:ineq-det-op-mon}, we prove that if $f$ is a non-constant positive-valued operator monotone function on $(0, \infty)$, $\Delta(f(A)) \le \Delta(f(\Phi(A)))$ with equality if and only if $\Phi(A) = A$. In fact, the result still holds for singular positive operators $A$ and positive-valued operator monotone functions on $[0, \infty)$ but the simple form of the equality condition is rendered ineffective in this scenario. Further in Theorem \ref{thm:ineq-det-unit-positive}, for a trace-preserving unital positive map $\Phi : \mathscr{R} \rightarrow \mathscr{R}$ and a continuous log-convex function $f$, we see that $\Delta(f(\Phi(A))) \le \Delta(f(A))$. As a corollary [Corollary \ref{cor:general-hadamard-unit-positive}], we obtain a version of Theorem \ref{thm:arveson-hadamard} for trace-preserving unital positive maps.

In \cite{matic}, Matic proves two inequalities, in the vein of Fischer's inequality, for the ratio $\frac{\det(A+D)}{\det A}$ to study the change in the determinant of a positive-definite matrix $A$ perturbed by a positive-definite block diagonal matrix $D$. For $k \in \mathbb{N}$, let $n, n_1, \cdots, n_k$ be positive integers such that $n = n_1 + \cdots +n_k$. For $i \in \langle k \rangle$, if $A_i$ is  in $M_{n_i}(\mathbb{C})$, we define an $n \times n$ matrix $\mathrm{diag}(A_1, \cdots, A_k)$ by,
$$\mathrm{diag}(A_1, \cdots, A_k )  := \begin{bmatrix}
A_1 & 0 & \cdots & 0\\
0 & A_2 & \cdots & 0\\
0 & 0 & \ddots & 0\\
0 & \cdots & 0 & A_k
\end{bmatrix}$$
With the above notation, we recall the two main inequalities from \cite{matic}.

\begin{thm}
\label{thm:det-ineq-matic-1}
{\textsl For each $i \in \langle k \rangle$, let $C_i, D_i$ be positive-definite matrices in $M_{n_i}(\mathbb{C})$. Let $C$ be a positive-definite matrix in block form in $M_n(\mathbb{C})$ with principal diagonal blocks given by $C_1, C_2, \cdots, C_k$. Then the following inequality holds, 
\begin{equation}
\label{eqn:det-ineq-matic-1}
\frac{\det(C+ \mathrm{diag}(D_1, \cdots, D_k))}{\det(C)} \ge \frac{\det(C_1 + D_1)}{\det(C_1)} \cdots \frac{\det(C_k + D_k)}{\det(C_k)}.
\end{equation} }
\end{thm}

\begin{thm}
\label{thm:det-ineq-matic-2}
{\textsl For each $i \in \langle k \rangle$, let $C_i, D_i$ be  positive-definite matrices in $M_{n_i}(\mathbb{C})$. Let $C$ be a  positive-definite matrix in block form in $M_n(\mathbb{C})$ such that the principal diagonal blocks of $C^{-1}$ (in block form) is given by $C_1 ^{-1}, C_2 ^{-1}, \cdots, C_k ^{-1}$. Then the following inequality holds, 
\begin{equation}
\label{eqn:det-ineq-matic-2}
\frac{\det(C+ \mathrm{diag}(D_1, \cdots, D_k))}{\det(C)} \le \frac{\det(C_1 + D_1)}{\det(C_1)} \cdots \frac{\det(C_k + D_k)}{\det(C_k)}.
\end{equation}  }
\end{thm}

As an application of the general framework developed, in this article we view the above inequalities as manifestations of Jensen's inequality in the context of conditional expectations on a finite von Neumann algebra for the choice of the operator monotone function $(1 + \frac{1}{x})^{-1}$ on $(0, \infty)$. Theorem \ref{thm:det-ineq-matic-1} and Theorem \ref{thm:det-ineq-matic-2} may be considered as specific cases of Theorem \ref{thm:ineq-det-matic} and Corollary \ref{cor:hadamard-cor1} respectively, which we state below. Not only does this provide us more insight but it also helps us directly identify the conditions under which equality holds. These equality conditions were not considered in \cite{matic}. In the two results mentioned below, $\mathscr{R}$ is a finite von Neumann algebra with a faithful normal tracial state $\tau$ and $\Phi$ is a $\tau$-preserving conditional expectation onto the von Neumann subalgebra $\mathscr{S}$ of $\mathscr{R}$.
\newline \newline
{\bf Theorem \ref{thm:ineq-det-matic} .} \textsl{For a regular positive operator $A$ in $\mathscr{R}$, and a positive operator $B$ in $\mathscr{S}$, the following inequality holds : 
\begin{equation}
\frac{\Delta(\Phi(A) + B))}{\Delta(\Phi(A))} \le \frac{\Delta(A + B)}{\Delta(A)}.
\end{equation}
If $B$ is regular, equality holds if and only if $\Phi(A) = A$ {\it i.e.}\ $A \in \mathscr{S}$.}
\newline \newline
{\bf Corollary \ref{cor:hadamard-cor1} .} \textsl{For a regular positive operator $A$ in $\mathscr{R}$, and a positive operator $B$ in $\mathscr{S}$, the following inequality holds :
\begin{equation}
\label{eqn:ineq-det-mat-intro}
\frac{\Delta(A + B)}{\Delta(A)} \le \frac{\Delta(\Phi(A^{-1})^{-1} + B)}{\Delta(\Phi(A^{-1})^{-1})},
\end{equation} 
with equality if and only if $B^{\frac{1}{2}}A^{-1}B^{\frac{1}{2}} \in \mathscr{S}$. In particular, if $B$ is regular, equality holds in (\ref{eqn:ineq-det-mat-intro}) if and only if $A \in \mathscr{S}$.}
\newline \newline
We note some improvements to Theorem \ref{thm:det-ineq-matic-1} and Theorem \ref{thm:det-ineq-matic-2} obtained from these generalizations. The inequalities (\ref{eqn:det-ineq-matic-1}) and (\ref{eqn:det-ineq-matic-2}) still hold if the matrices $D_1, \cdots, D_k$ are positive-semidefinite. If the matrices $D_1, \cdots, D_k$ are positive-definite, equality holds in (\ref{eqn:det-ineq-matic-1}) and (\ref{eqn:det-ineq-matic-2}) if and only if $C = \mathrm{diag}(C_1, \cdots, C_k)$ \textit{i.e.}\ $C$ is in block diagonal form. If the matrices $D_1, \cdots, D_k$ are positive-semidefinite and $D := \textrm{diag}(D_1, \cdots, D_k)$, equality holds in inequality (\ref{eqn:det-ineq-matic-2}) if and only if $D^{\frac{1}{2}}C^{-1}D^{\frac{1}{2}}$ is in block diagonal form. Note that when $D$ is positive-definite, $D^{\frac{1}{2}}C^{-1}D^{\frac{1}{2}}$ is in block diagonal form if and only if $C$ is in block diagonal form. Substantially more effort goes into proving Theorem \ref{thm:ineq-det-matic} as compared to Corollary \ref{cor:hadamard-cor1}.

A lot has been said in the literature about operator theoretic versions of Jensen's inequality. For instance, using results by Davis(\cite{davis}), Choi(\cite{choi}), one obtains a general operator theoretic version of Jensen's inequality involving operator convex functions and unital positive maps between $C^*$-algebras. For a continuous function $f$ on an interval $I \subseteq \mathbb{R}$,  if $f(\sum_{i=1}^k A_i T_i A_i ^*) \le \sum_{i=1}^k A_i f(T_i) A_i ^*$, for self-adjoint operators $T_1, \cdots, T_k$ on an infinite-dimensional Hilbert space $\mathscr{H}$ with spectra in $I$, and operators $A_1, \cdots A_k$ on $\mathscr{H}$ such that $\sum_{i=1}^k A_i A_i ^* = I$, we say that $f$ is operator $C^*$-convex.  In \cite{han-ped-jen}, Hansen and Pedersen prove the equivalence of the class of operator convex functions on an interval $I$ and the class of operator $C^*$-convex functions on $I$. Our intention is not to reiterate results from the literature but rather to adapt the key ideas wherever necessary to increase our understanding of determinant inequalities.

This article is organized as follows. In section \ref{sec:prelim}, we setup the background for our discussion. Along with setting up notation, we offer a primer on the theory of von Neumann algebras, the Fuglede-Kadison determinant, conditional expectations on von Neumann algebras, and the theory of operator convex and operator monotone  functions on $[0, \infty)$. In section \ref{sec:use-lemma}, we collect some technical lemmas. The crux of the discussion is in section \ref{sec:detineq1} where extensions of Hadamard's inequality involving the Fuglede-Kadison determinant are proved. In the final section, we discuss applications of the results derived in section \ref{sec:detineq1} and obtain the aforementioned inequalities as special cases.

\subsection*{Acknowledgments} This work is an extension of the research done in the author's doctoral dissertation at the University of Pennsylvania, Philadelphia. The author extends his immense gratitude to his doctoral advisor, Richard Kadison, for his helpful comments regarding this manuscript, especially in regard to the historical context of the Hadamard inequality and its generalizations, and bringing to his attention, Jensen's operator inequality described in \cite{han-ped-jen} by Hansen and Pedersen. The author also thanks Raghavendra Venkatraman and Maxim Gilula for their careful reading of an early draft and their suggestions on the organization and exposition of the material. A token of gratitude is due to an anonymous referee whose keen observations and comments helped clarify certain points in the article.

%%%%%%%%%%%%%%%%%%%%%%%%%%%%%%%%%
%
%  PRELIMINARIES
%
%%%%%%%%%%%%%%%%%%%%%%%%%%%%%%%%%

\section{Preliminaries}
\label{sec:prelim}
% ask Kadison about capitalization of 'v'in von 
\subsection{Von Neumann Algebras}
For us, $\mathscr{H}$ denotes a separable complex Hilbert space with inner product $\langle \cdot , \cdot \rangle$, and $\mathcal{B}(\mathscr{H})$ is the set of bounded operators from $\mathscr{H}$ to $\mathscr{H}$. We denote the identity operator in $\mathcal{B}(\mathscr{H})$ by $I$ {\it i.e.}\ $Ix = x$ for all $x$ in $\mathscr{H}$. The space $\mathcal{B}(\mathscr{H})$ may be considered as a Banach algebra with the operator norm, and multiplication given by composition of operators. In addition, with the adjoint operation as involution ($T \rightarrow T^*$), it is also a $C^*$-algebra. All norm-closed $*$-subalgebras of $\mathcal{B}(\mathscr{H})$ are also $C^*$-algebras. A positive linear functional $\rho$ on a unital $C^*$-algebra is said to be a {\it state} if $\rho(I) = 1$. A state $\rho$ is said to be faithful if $\rho(A^*A) = 0$ if and only if $A=0$. There are several interesting topologies on $\mathcal{B}(\mathscr{H})$ coarser than the norm topology. Two important ones are {\it the weak-operator topology}, which is the coarsest topology such that for any vectors $x, y $ in $\mathscr{H}$ the functional $\rho_{x, y} : \mathcal{B}(\mathscr{H}) \rightarrow \mathbb{C}$ defined by $\rho_{x,y}  (T) = \langle T x, y \rangle$ is continuous, and the {\it strong-operator topology}, which is the coarsest topology such that for any vector $x$ in $\mathcal{B}(\mathscr{H})$ the map $\omega_x : \mathcal{B}(\mathscr{H}) \rightarrow \mathbb{C}$ defined by $\omega_x(T) = Tx$ is norm-continuous. The commutant of a non-empty subset $\mathscr{F}$ of $\mathcal{B}(\mathscr{H})$ is defined as $\mathscr{F}' := \{ T : AT = TA, \forall A \in \mathscr{F} \}.$ Before we define von Neumann algebras, we recall the von Neumann double commutant theorem (\cite{double-commutant}) to put the algebraic and analytic aspects of von Neumann algebras into perspective, the commutant being an algebraic object, and the topologies considered determining the analytic aspect.
\begin{thm}[double commutant theorem]
\label{thm:bicommutant}
Let $\mathfrak{A}$ be a self-adjoint subalgebra of $\mathcal{B}(\mathscr{H})$ containing the identity operator. Then the following are equivalent:
\begin{itemize}
\item[(i)] $\mathfrak{A}$ is weak-operator closed,
\item[(ii)] $\mathfrak{A}$ is strong-operator closed,
\item[(iii)] $(\mathfrak{A}')' = \mathfrak{A}$. 
\end{itemize}
\end{thm}
\begin{definition}
A von Neumann algebra is a self-adjoint subalgebra of $\mathcal{B}(\mathscr{H})$ containing the identity operator which is closed under the weak operator topology.
\end{definition}
The inspiration to study von Neumann algebras comes from, amongst other things, the study of group representations on infinite-dimensional Hilbert spaces. They were introduced in \cite{ring-op-2} by Murray and von Neumann, as rings of operators. Von Neumann algebras have plenty of projections, in the sense that, the set of linear combinations of projections in a von Neumann algebra $\mathscr{R}$ is norm-dense in $\mathscr{R}$. This is clear from the spectral theorem for self-adjoint operators and the observation that any operator $T$ can be written as a linear combination of self-adjoint operators ($T = \frac{T+T^*}{2} + \textrm{i} \frac{T-T^*}{2\textrm{i}}$.) 

In a bid to classify von Neumann algebras, Murray and von Neumann developed the comparison theory of projections. Two projections $E, F$ in $\mathscr{R}$ are said to be equivalent if there is an operator $V \in \mathscr{R}$ such that $VV^* = E$ and $V^*V = F$, and such a $V$ is called a partial isometry. In general, it is possible to have equivalent projections $E, F$ such that $E \le F$ and $E \ne F$. If $\mathscr{R}$ does not allow for such occurrences, it is said to be a finite von Neumann algebra. Alternatively, the characterizing property for a finite von Neumann algebra is that every isometry is a unitary operator {\it i.e.}\ for $V \in \mathscr{R}$ if $V^*V = I$, then $VV^* = I$. A major part of their study involved classifying the so-called {\it factors}, which are von Neumann algebras with trivial center, the set of scalar multiples of the identity. In a nutshell, factors may be thought of as building blocks of von Neumann algebras, and finite factors may be thought of as building blocks of finite von Neumann algebras. Finite factors come in two flavors : the finite-dimensional kind, given by $M_n(\mathbb{C}), n \in \mathbb{N}$, and the infinite-dimensional kind, the $II_1$ factors. A characterizing property of a finite factor is the existence of a unique faithful normal tracial state {\it i.e.}\ a linear functional $\tau$ satisfying the following conditions : (i) $\tau(AB) = \tau(BA)$ for all $A, B \in \mathscr{R}$, (ii) $\tau(I) = 1$, (iii) (faithfulness) $\tau(A^*A) = 0$ if and only if $A=0$, and (iv) (normality) for an increasing sequence of projections $\{ E_n \}$ we have $\tau(\sup_{n \in \mathbb{N}} E_n) = \sup _{n \in \mathbb{N}} \tau(E_n)$. A von Neumann algebra $\mathscr{R}$ has a faithful normal tracial state if and only if $\mathscr{R}$ is finite. We direct the reader to chapters 5-10 in \cite{kadison-ringrose} for a fuller discussion of this topic.

\subsection{Fuglede-Kadison Determinant}
In \cite{fk-det}, for a finite factor $\mathscr{M}$ with the unique faithful normal tracial state $\tau$, Fuglede and Kadison define the notion of a determinant on $GL_1(\mathscr{M})$, the set of regular operators in $\mathscr{M}$, in the following manner :
$$\Delta : GL_1(\mathscr{M}) \rightarrow \mathbb{R}_{+}, $$
$$\Delta(A) = \exp(\tau(\log (A^*A)^{\frac{1}{2}})).$$
That $\Delta$ makes sense is a consequence of the continuous functional calculus. Several properties of $\Delta$ are studied in \cite{fk-det} and a major portion of the effort goes in proving that it is a group homomorphism. Further, it is proved that for any `reasonable' determinant theory, the determinant of an operator with a non-trivial nullspace must vanish. But in general, there can be other extensions of the definition to singular operators with trivial nullspace. Two such extensions are mentioned; the \textit{algebraic extension} defines the determinant to be zero for all singular operators in $\mathscr{M}$, whereas the \textit{analytic extension} uses the spectral decomposition of $(A^*A)^{\frac{1}{2}} (= \int \lambda \; dE_{\lambda})$ to define $\Delta(A) := \exp(\int \log \lambda \; d\tau(E_{\lambda}))$ with the understanding that $\Delta(A) = 0$ if $\int \log \lambda \; d\tau(E_{\lambda}) = -\infty$. 

Although originally in \cite{fk-det}, the Fuglede-Kadison determinant is defined for finite factors, the discussion carries over to the case of a von Neumann algebra with a tracial state $\tau$. Our primary interest is in the case when the tracial state $\tau$ is faithful and thus, $\mathscr{R}$ is finite. As our results are true for any choice of a faithful normal tracial state, the dependence of $\Delta$ on $\tau$ will be suppressed in the notation. We direct the reader to \cite{harpe} for a masterful account of the Fuglede-Kadison determinant (and its variants) by de la Harpe.

\begin{remark}
Let $\mathscr{R}$ be a finite von Neumann algebra with identity $I$ and a faithful normal tracial state $\tau$. We will consider the analytic extension of Fuglede-Kadison determinant and by abuse of notation, we denote it also by $\Delta$. Below we note some useful properties of $\Delta$ that are pertinent to our discussion. 
\begin{itemize}
\item[(i)] $\Delta(I) = 1$.
\item[(ii)] $\Delta(AB) = \Delta(A) \Delta(B),$ for $A, B$ in $\mathscr{R}$.
\item[(iii)] $\Delta(A^{-1}) = \Delta(A)^{-1},$ for regular $A$ in $\mathscr{R}$.
\item[(iv)] $\Delta$ is continuous on $GL_1(\mathscr{R})$, in the uniform topology.
\item[(v)] $\Delta(A) \le \Delta(B)$,  for $A, B$ in $\mathscr{R}$ such that $0 \le A \le B$ with $A$ being regular, and equality holds if and only if $A=B$.
\item[(vi)] $\lim_{\varepsilon \rightarrow 0^{+}}\Delta(A+\varepsilon I) = \Delta(A),$ for positive $A$ in $\mathscr{R}$.
\end{itemize}
Note that  (v) follows from the fact that $\log$ is an operator monotone function on $(0, \infty),$ (thus, $\log A \le \log B$) and by the faithfulness of the tracial state, we have that $\Delta(A) = \Delta(B) \Leftrightarrow \tau(\log A) = \tau(\log B) \Leftrightarrow \tau(\log B - \log A) = 0 \Leftrightarrow \log A = \log B \Leftrightarrow A = B.$ We will have more to say about operator monotone functions in subsection \ref{sec:lowner}.
\end{remark}

\begin{example}[Fuglede-Kadison determinant for $M_n(\mathbb{C})$]
Let us denote the usual determinant function on $M_n(\mathbb{C})$ by $\det$, and the normalized trace on $M_n(\mathbb{C})$ by $\mathrm{tr}$, defined as the average of the diagonal entries of the matrix. Later when we want to emphasize $n$, we will denote the determinant, normalized trace on $M_n(\mathbb{C})$ by $\mathrm{det}_{\mathrm{n}}, \mathrm{tr}_{\mathrm{n}}$, respectively. For a matrix $A$ in $M_n(\mathbb{C})$, if $\lambda_1, \cdots, \lambda_n$ are the eigenvalues of the positive-definite matrix $(A^*A)^{\frac{1}{2}}$ (counted with multiplicity), we have that, $$\mathrm{tr}(\log (A^*A)^{\frac{1}{2}}) = \frac{\log \lambda_1 + \cdots + \log \lambda_n}{n} = \log  (\sqrt[n]{\lambda_1 \cdots \lambda_n}) = \log (\sqrt[n]{(\det (A^*A)^{\frac{1}{2}})})$$
Thus we see that for the type $I_n$ factor $M_n(\mathbb{C})$ ($n \in \mathbb{N}$), we have the following relationship between $\Delta$ and $\det$ :   
\begin{equation}
\Delta(A) = \sqrt[n]{(\det(A^*A)^{\frac{1}{2}})} = \sqrt[n]{|\det A|},
\end{equation} 
for any $A$ in $M_n(\mathbb{C})$.

\end{example}

\subsection{Conditional Expectations}
Let $\mathscr{R}$ denote a von Neumann algebra with identity $I$. Let $\mathscr{S}$ denote a von Neumann subalgebra of $\mathscr{R}$. Then a map $\Phi : \mathscr{R} \rightarrow \mathscr{S}$ is said to be a {\it conditional expectation} from $\mathscr{R}$ onto $\mathscr{S}$ if it satisfies the following :
\begin{itemize}
\item[(i)] $\Phi$ is linear, positive and $\Phi(I) = I,$
\item[(ii)] $\Phi(S_1RS_2) = S_1 \Phi(R) S_2$ for $R$ in $\mathscr{R}$, and $S_1, S_2$ in $\mathscr{S}$.
\end{itemize}
From (ii), we have that $\Phi(T) = T$ if and only if $T$ is in $\mathscr{S}$.

For a finite von Neumann algebra $\mathscr{R}$ with a faithful normal tracial state $\tau$, a map $\Phi : \mathscr{R} \rightarrow \mathscr{R}$ is said to be \emph{$\tau$-preserving} or \emph{trace-preserving} if $\tau(\Phi(A)) = \tau(A)$ for $A$ in $\mathscr{R}$. In this article, we are primarily interested in trace-preserving conditional expectations on finite von Neumann algebras.
\begin{example}
\label{ex:cond-exp-1}
Let $D_n(\mathbb{C})$ denote the subalgebra of $M_n(\mathbb{C})$ consisting of diagonal matrices. Define a map $\Phi : M_n(\mathbb{C}) \rightarrow D_n(\mathbb{C})$ by $\Phi(A) := \mathrm{diag}(a_{11}, \cdots, a_{nn})$. The map $\Phi$ is a trace preserving conditional expectation from the finite von Neumann algebra $M_n(\mathbb{C})$ onto $D_n(\mathbb{C})$. 

\begin{comment}
That $\Phi$ is linear, trace preserving and takes the identity matrix to itself is straightforward to see. As the diagonal elements of a positive-definite matrix are positive, $\Phi$ is also a positive map. Let $D := \mathrm{diag}(d_1, \cdots, d_n)$ be an element of $ D_n(\mathbb{C})$. To check (ii), note that the $(i, i)$ entry of $DA$ is $d_i a_{ii}$, which is the same as the $(i, i)$ entry of the diagonal matrix $D\Phi(A)$. Thus $\Phi(DA) = D\Phi(A)$ and similarly, $\Phi(AD) = \Phi(A)D$.
\end{comment} 
\end{example}

\begin{example}[Conditional expectations on $M_n(\mathbb{C})$]
\label{ex:cond-exp-2}
Let $n = n_1 + \cdots + n_k$, where $n, n_1, \cdots, n_k \in \mathbb{N}$. Note that for matrices $A_1, \cdots, A_k$ in $M_{n_1}(\mathbb{C}), \cdots , M_{n_k}(\mathbb{C})$ respectively, we may construct a matrix $A$ in $M_n(\mathbb{C})$ with these matrices as the principal diagonal blocks and $0$'s elsewhere, \emph{i.e.} $A := \mathrm{diag}(A_1, \cdots, A_k).$ In this manner, one may consider the matrix algebra $M_{n_1}(\mathbb{C}) \oplus \cdots \oplus M_{n_k}(\mathbb{C})$ as a subalgebra of $M_n(\mathbb{C})$ with the same identity. Consider
the map $\Phi : M_n(\mathbb{C}) \rightarrow M_{n_1}(\mathbb{C}) \oplus \cdots \oplus M_{n_k}(\mathbb{C}),$ defined by $\Phi(A) = A_{11} \oplus \cdots \oplus A_{kk}$ where $A_{ii}$'s are the principal $n_i \times n_i$ diagonal blocks of $A$. It is left to the reader to check that $\Phi$ is a trace-preserving conditional expectation from $M_n(\mathbb{C})$ onto $M_{n_1}(\mathbb{C}) \oplus \cdots \oplus M_{n_k}(\mathbb{C})$. 

\end{example}

A positive linear map $\Psi : \mathscr{R} \rightarrow \mathscr{R}$ is said to be {\it $n$-positive} (for $n \in \mathbb{N}$) if $\Psi \otimes I_n : \mathscr{R} \otimes M_n(\mathbb{C}) \rightarrow \mathscr{R} \otimes M_n(\mathbb{C})$ is positive. Further if $\Psi$ is $n$-positive for all $n$ in $\mathbb{N}$, we say that $\Psi$ is {\it completely positive}. We mention without proof the following two theorems about conditional expectations on finite von Neumann algebras that play a fundamental role in section 4.

\begin{thm}[see \cite{umegaki}, \cite{tomiyama}, {\cite[Theorem 7]{cond-expect-kad}}] 
{\textsl Let $\mathscr{R}$ be a finite von Neumann algebra with a faithful normal tracial state $\tau$ and $\mathscr{S}$ be a von Neumann subalgebra of $\mathscr{R}$. Then there is a unique map $\Phi : \mathscr{R} \rightarrow \mathscr{S}$ such that $\tau(\Phi(R)S) = \tau(RS)$ for $R\in \mathscr{R}, S \in \mathscr{S}$, and such a map $\Phi$ is a trace-preserving normal conditional expectation from $\mathscr{R}$ onto $\mathscr{S}$.}

\end{thm}

\begin{thm}[{\cite[Theorem 1]{cond-exp}}]
\label{thm:cond-exp-comp-positive}
{\textsl Let $\mathscr{R}$ be a finite von Neumann algebra with a faithful normal tracial state $\tau$ and $\mathscr{S}$ be a von Neumann subalgebra of $\mathscr{R}$. Then the trace-preserving normal conditional expectation from $\mathscr{R}$ onto $\mathscr{S}$ is a completely positive map.}
\end{thm}

\subsection{Operator Monotone and Operator Convex Functions}
\label{sec:lowner}

 A continuous function $f$ defined on the interval $\Gamma \subseteq \mathbb{R}$ is said to be {\it operator monotone} if for self-adjoint operators $A, B$ on an infinite-dimensional Hilbert space $\mathscr{H}$, with spectra in $\Gamma$, such that $A \le B$, we have that $f(A) \le f(B)$. For positive operators $A, B$ with $A \le B$ it is not necessarily true that $A^2 \le B^2$. But for $0 < r \le 1$, the L{\"o}wner-Heinz inequality states that $A^r \le B^r$ for $0 < r \le 1$. Thus $x^r$ is an operator monotone function on $[0, \infty)$ for $0 < r \le 1$. Other examples include $\log (1+x)$ on $[0, \infty)$, $\log x$ on $(0, \infty)$. In his seminal paper \cite{lowner}, L{\"o}wner studied operator monotone functions in detail establishing their relationship with a class of analytic functions called Pick functions. We state an integral representation of operator monotone functions on $[0, \infty)$ (see \cite[Chapter V]{bhatia-ma}).
\begin{thm}
\label{thm:int-op-mon}
{\textsl A continuous real-valued function $f$ on $[0, \infty)$ is operator monotone if and only if there is a finite positive measure $\mu$ on $(0, \infty)$ such that $$f(t) = a+ bt + \int_{0}^{\infty} \frac{ (\lambda + 1)t}{\lambda + t} \; d\mu(\lambda), t \in [0, \infty), $$ for some real number $a$, and $b \ge 0$.}
\end{thm}

The concept of operator monotonicity is closely related to operator convex functions which were studied by Kraus (\cite{kraus}). Let $f$ be a continuous function on the interval $\Gamma \subseteq \mathbb{R}$. We say that $f$ is {\it operator convex} if $f(\lambda A  + (1- \lambda) B) \le \lambda f(A) + (1 - \lambda) f(B)$, for each $\lambda$ in $[0, 1]$, and self-adjoint operators $A, B$ on an infinite-dimensional Hilbert space $\mathscr{H}$ with spectra in $\Gamma$. We state an integral representation of operator convex functions on $[0, \infty)$ for which $f'(0^{+})$ exists, which may be derived from Theorem \ref{thm:int-op-mon} above and \cite[Theorem 2.4]{hansen-pedersen}.

\begin{thm}
\label{thm:int-op-conv}
{\textsl A continuous real-valued function $f$ on $[0, \infty)$ such that $f'(0^{+})$ exists is operator convex if and only if there is a finite positive measure $\mu$ on $(0, \infty)$ such that $$f(t) = a + bt + ct^2 + \int_{0}^{\infty} \frac{(\lambda + 1) t^2}{\lambda + t} \; d\mu(\lambda), t \in [0, \infty),$$ for some real numbers $a, b,$ and $c \ge 0$.}
\end{thm}

\section{A Collection of Useful Lemmas}
\label{sec:use-lemma}
In this section we gather a collection of disparate results which serve us well in section \ref{sec:detineq1}.
\begin{definition}
A real-valued convex function on the interval $[a, b] \subseteq \mathbb{R}$ is said to be \textit{strictly convex} if for $x, y \in [a, b]$ such that $f(\frac{x+y}{2}) = \frac{f(x) + f(y)}{2}$, we must have $x=y$. 
\end{definition}

\begin{lem}[Jensen's inequality]
\textsl{ For a convex function $f$ on the real interval $[a,b]$ and a probability measure $\mu$ on a compact subset $S$ of $[a, b]$, we have the following inequality, $$f(\int_{S} \lambda \; d\mu(\lambda)) \le  \int_{S} f(\lambda) \; d\mu(\lambda).$$ Further if $f$ is strictly convex, then equality holds if and only if $\mu$ is supported on a point {\it i.e.}\ $\mu$ is a Dirac measure. }
\end{lem}

\begin{lem}
\label{lem:jensen-state}
\textsl{ Let $\mathfrak{A}$ be a unital $C^*$-algebra and $\rho$ be a state on $\mathfrak{A}$. For a continuous convex function $f$ on a real interval $[a, b]$ and a self-adjoint operator $A$ in $\mathfrak{A}$ with spectrum in $[a, b]$, we have the following inequality, $$f(\rho(A)) \le \rho(f(A)).$$ Further if $\rho$ is faithful and $f$ is strictly convex, then equality holds if and only if $A$ is a scalar multiple of the identity.}
\end{lem}
\begin{proof}
The unital $C^*$-algebra generated by $A$ is $*$-isomorphic to $C(\sigma(A))$ and henceforth referred to interchangeably. Note that the restriction of $\rho$ to $C(\sigma(A))$ is also a state, and faithful if $\rho$ is faithful. By the Riesz representation theorem, there is a probability measure $\mu$ on $\sigma(A)$ such that for any continuous function $f$ on $\sigma(A)$, we have that $\rho(f(A)) = \int_{\sigma(A)} f(\lambda) \; d\mu(\lambda) $. Note that on the compact subset $\sigma(A)$ of $[a, b]$, by Jensen's inequality, $$f (\rho(A)) = f(\int_{\sigma (A)} \lambda \; d\mu (\lambda) ) \le \int_{\sigma(A)} f(\lambda) \; d\mu(\lambda) = \rho(f(A)).$$
If $f$ is strictly convex, equality holds if and only if $\mu$ is a Dirac measure, say, supported on $\lambda ' \in [a, b]$. In addition if $\rho$ is a faithful state, then $\mu$ is a Dirac measure supported on $\{ \lambda ' \}$ if and only if $\sigma(A) = \{\lambda'\} \Leftrightarrow A = \lambda 'I.$
\end{proof}

\begin{lem}
\label{lem:det-trace-jensen}
\textsl{For a finite von Neumann algebra $\mathscr{R}$ with a tracial state $\tau$ and a positive operator $A$ in $\mathscr{R}$, we have the following inequality,
\begin{equation}
\Delta(A) \le \tau(A).
\end{equation}
If $A$ is a regular positive operator and $\tau$ is faithful, equality holds if and only if $A$ is a positive scalar multiple of the identity $I$.}
\end{lem}
\begin{proof}
We first prove the inequality for a regular positive operator $A$ in $\mathscr{R}$. Note that the spectrum of $A$ is contained in $[\|A^{-1}\|^{-1}, \|A\|]$. The function $-\log x$ defined on $[\|A^{-1}\|^{-1}, \|A\|]$ is strictly convex as the second derivative is $\frac{1}{x^2}$ which is strictly positive on $[\|A^{-1}\|^{-1}, \|A\|]$. From Lemma \ref{lem:jensen-state} for $-\log$, we have that $-\log \tau(A) \le \tau(- \log A) \Rightarrow \tau( \log A) \le \log \tau(A) \Rightarrow \Delta(A) = \exp (\tau( \log A)) \le \exp(\log \tau(A)) = \tau(A)$, and if $\tau$ is faithful, equality holds if and only if $A$ is a positive scalar multiple of the identity. 

We next prove the inequality for a singular positive operator $A$. For $\varepsilon > 0$, note that $A + \varepsilon I$ is a regular positive operator. Thus we have that $\Delta(A+\varepsilon I) \le \tau(A+\varepsilon I)$. Using the norm-continuity of $\tau$, and taking the limits as $\varepsilon \rightarrow 0^{+}$, we get that $\Delta(A) \le \tau(A)$.
\end{proof}

In this paragraph, we set up the notation to be used in the proof of Lemma \ref{lem:sylvester-rule}. Let $\mathscr{R}$ be a von Neumann algebra, acting on the Hilbert space $\mathscr{H}$, with a tracial state $\tau$. For the von Neumann algebra $M_2(\mathscr{R}) \cong \mathscr{R} \otimes M_2(\mathbb{C})$ (acting on $\mathscr{H} \oplus \mathscr{H}$), we are interested in the tracial state on $M_2(\mathscr{R})$ given by $\tau_2 = \tau \otimes \mathrm{tr}_2$ {\it i.e.} for an operator $A$ in $M_2(\mathscr{R}), \tau_2(A) = \frac{\tau(A_{11}) + \tau(A_{22})}{2}$ where $A_{ij} \in \mathscr{R}$ ($1 \le i, j \le 2$) denotes the $(i, j)^{\mathrm{th}}$ entry of $A$.  We denote the Fuglede-Kadison determinant on $M_2(\mathscr{R})$ corresponding to $\tau \otimes \mathrm{tr}_2$ by $\Delta_2$. For operators $A_1, A_2$ in $\mathscr{R}$, we define 
$$\mathrm{diag}(A_1, A_2) := \begin{bmatrix}
A_1 & 0\\
0 & A_2
\end{bmatrix} \in M_2(\mathscr{R}).$$ 

\begin{remark}
\label{rmrk:high-det}
For operators $A_1, A_2$ in $\mathscr{R}$, we have that $\Delta_2(\mathrm{diag}(A_1, A_2)) = \sqrt{\Delta(A_1) \cdot \Delta(A_2)}$.
\end{remark}

\begin{lem}[Sylvester's identity for $\Delta$]
\label{lem:sylvester-rule}
\textsl{Let $\mathscr{R}$ be a von Neumann algebra with a tracial state $\tau$. For operators $A, B$ in $\mathscr{R}$, we have that $\Delta(I+AB) = \Delta(I + BA).$
}
\end{lem}
\begin{proof}
For a unital ring $R$ with multiplicative identity $1$ and an element $x$ in $M_2(R)$, we have that
$$\left[ \begin{array}{ccc}
1 &  0 \\
x & 1 \end{array} \right]
\left[ \begin{array}{ccc}
1 &  0 \\
-x & 1 \end{array} \right] = 
\left[ \begin{array}{ccc}
1 &  0 \\
0 & 1 \end{array} \right] ,
\left[ \begin{array}{ccc}
1 & x \\
0 & 1 \end{array} \right]
\left[ \begin{array}{ccc}
1 & -x \\
0 & 1 \end{array} \right] = 
\left[ \begin{array}{ccc}
1 &  0 \\
0 & 1 \end{array} \right]
 $$
Thus the operators $$\left[ \begin{array}{ccc}
I &  0 \\
-B & I \end{array} \right], 
\left[ \begin{array}{ccc}
I & A \\
0 & I\end{array} \right] \in M_2(\mathscr{R})$$ are regular and their Fuglede-Kadison determinant is strictly positive. Using the multiplicativity of $\Delta_2$, Remark \ref{rmrk:high-det} and the algebraic identity given below
\begin{align*}
\left[ \begin{array}{ccc}
I &  A \\
-B & I \end{array} \right]\ &=\ \left[ \begin{array}{ccc}
I &  0 \\
-B & I \end{array} \right]\ \left[ \begin{array}{ccc}
I & 0 \\
0 & I + BA \end{array} \right]\ \left[ \begin{array}{ccc}
I & A \\
0 & I \end{array} \right]\\
 &= \ \left[ \begin{array}{ccc}
I &  A \\
0 & I \end{array} \right]\ \left[ \begin{array}{ccc}
I+AB & 0 \\
0 & I \end{array} \right]\ \left[ \begin{array}{ccc}
I & 0 \\
-B & I \end{array} \right],
\end{align*}
we conclude that $\Delta(I + AB) = \Delta(I + BA)$.
\end{proof}

In Lemma \ref{lem:positive-two-tensor} below, we adapt Theorem 1.3.3 in \cite{bhatia-pdm} which involves positive-definite matrices to the context of positive operators on a Hilbert space by mimicking the algebraic trick used therein. 

\begin{lem}
\label{lem:positive-two-tensor}
\textsl{Let $A, C$ be positive operators in $\mathcal{B}(\mathscr{H})$ with $A$ being regular. Let $B$ be an operator in $\mathcal{B}(\mathscr{H})$. Then the self-adjoint operator, $$\mathbf{P} := 
\begin{bmatrix}
A & B\\
B^* & C
\end{bmatrix}$$ in $\mathcal{B}(\mathscr{H} \oplus \mathscr{H})$ is positive if and only if the Schur complement $C - B^*A^{-1}B$ is positive.
}
\end{lem}
\begin{proof}
Consider the operator in $\mathcal{B}(\mathscr{H} \oplus \mathscr{H})$ given by,
$$ \mathbf{X} := \left[ \begin{array}{ccc}
I & 0 \\
-B^*A^{-1} & I \end{array} \right] .
$$
As $\mathbf{X}$ is regular, $\mathbf{P}$ is positive if and only if $\mathbf{X} \mathbf{P} \mathbf{X}^*$ is positive. A straightforward matrix computation shows that 
$$\mathbf{X} \mathbf{P} \mathbf{X}^* = \left[ \begin{array}{ccc}
A &  0 \\
0 & C - B^*A^{-1}B \end{array} \right] .$$
Thus, $\mathbf{P}$ is positive $\Leftrightarrow \mathbf{X} \mathbf{P} \mathbf{X}^*$ is positive $\Leftrightarrow C - B^*A^{-1}B$ is positive.
\end{proof}

We paraphrase Proposition 3.8 from \cite{jensen} below without proof.
\begin{thm}
\label{thm:jensen-unital-positive}
{\textsl Let $\mathfrak{A}, \mathfrak{B}$ be unital $C^*$-algebras, such that $\mathfrak{B}$ is abelian, and let $\Phi : \mathfrak{A} \rightarrow \mathfrak{B}$ be a unital positive map. Let $f$ be a real-valued continuous convex function defined on the interval $[a, b]$. For a self-adjoint operator $A$ of $\mathfrak{A}$ whose spectrum is contained in $[a, b]$, we have the following inequality:
\begin{equation}
f(\Phi(A)) \le \Phi(f(A)).
\end{equation}
}
\end{thm}

\section{Some Determinant Inequalities}
\label{sec:detineq1}
Throughout this section, $\mathscr{R}$ will denote a {\bf finite} von Neumann algebra acting on the Hilbert space $\mathscr{H}$ with identity $I$ and a faithful normal tracial state $\tau$, $\mathscr{S}$ will denote a von Neumann subalgebra of $\mathscr{R}$, and $\Phi$ will denote a $\tau$-preserving conditional expectation from $\mathscr{R}$ onto $\mathscr{S}$. Note that for $A$ in $\mathscr{R}$, if $\varepsilon I \le A$, then $\varepsilon I \le \Phi(A)$. As a consequence, if $A$ is a positive regular operator, $\Phi(A)$ is also positive and regular.

\begin{thm}
\label{thm:general-hadamard}
\textsl{For any regular positive operator $A$ in $\mathscr{R}$, we have that 
\begin{equation}
\label{eqn:hadamard}
\Delta(\Phi(A^{-1})^{-1}) \le \Delta(A) \le \Delta(\Phi(A)). 
\end{equation}
with equality on either side if and only if $\Phi(A) = A$ \textit{i.e} $A \in \mathscr{S}$. If $A$ is positive (but not necessarily regular), one still has the inequality on the right {\it i.e.}\ $\Delta(A) \le \Delta(\Phi(A)).$ } 
\end{thm}

\begin{proof}
Let $A$ be a regular positive operator in $\mathscr{R}$. As $\Phi(A)^{-1}$ is in $\mathscr{S}$, using Lemma \ref{lem:det-trace-jensen} for the regular positive operator $\Phi(A)^{-\frac{1}{2}}A\Phi(A)^{-\frac{1}{2}}$, and keeping in mind the trace-preserving nature of $\Phi$, we note that,
\begin{align*}
 \Delta(\Phi(A)^{-\frac{1}{2}}A\Phi(A)^{-\frac{1}{2}}) & \le \tau(\Phi(A)^{-\frac{1}{2}}A\Phi(A)^{-\frac{1}{2}}) = \tau(A \Phi(A)^{-1}) \\
  & =\tau(\Phi(A \Phi(A)^{-1})) = \tau(\Phi(A)\Phi(A)^{-1}) \\
  & = \tau(I) = 1,
\end{align*}
with equality if and only if $\Phi(A)^{-\frac{1}{2}}A\Phi(A)^{-\frac{1}{2}} = I \Leftrightarrow \Phi(A) = A$.

Using the multiplicativity of $\Delta$, we prove the desired inequality below.
\begin{align}
\label{eqn:general-hadamard-2}
  &\phantom{{}\Rightarrow{}} \Delta(\Phi(A)^{-\frac{1}{2}} A\Phi(A)^{-\frac{1}{2}}) \le 1 \notag \\
  & \Rightarrow \Delta(\Phi(A))^{-\frac{1}{2}}\Delta(A)\Delta(\Phi(A))^{-\frac{1}{2}} \le 1 \notag\\
  & \Rightarrow \Delta(A) \le \Delta(\Phi(A))^{\frac{1}{2}}\Delta(\Phi(A))^{\frac{1}{2}} = \Delta(\Phi(A)),
\end{align}
with equality if and only if $\Phi(A) = A$. Using the inequality just proved for the regular positive operator $A^{-1}$, we have $\Delta(A^{-1}) \le \Delta(\Phi(A^{-1})) \Leftrightarrow \Delta(\Phi(A^{-1})^{-1}) \le \Delta(A)$, with equality if and only if $\Phi(A^{-1}) = A^{-1}$. Note that $\Phi(A^{-1}) = A^{-1} \Leftrightarrow A^{-1} \in \mathscr{S} \Leftrightarrow A \in \mathscr{S} \Leftrightarrow \Phi(A) = A$. 

Let $\varepsilon > 0.$ If $A$ is positive but not necessarily regular, applying inequality (\ref{eqn:general-hadamard-2}) to the regular operator $A+\varepsilon I$ yields the following inequality $$\Delta(A+\varepsilon I) \le \Delta(\Phi(A + \varepsilon I)) = \Delta(\Phi(A) + \varepsilon I).$$
Taking the limit as $\varepsilon \rightarrow 0^{+}$, we see that $\Delta(A) \le \Delta(\Phi(A)).$

\end{proof}

\begin{remark}
\label{rmrk:general-hadamard}
In the proof of Theorem \ref{thm:general-hadamard}, we used the faithfulness of the tracial state $\tau$ to invoke the equality condition in Lemma \ref{lem:det-trace-jensen}. But inequality (\ref{eqn:hadamard}) still holds in any von Neumann algebra with a tracial state $\tau$ which may be neither faithful nor normal. This observation is precisely the generalization of Hadamard's inequality by Arveson as in Corollary 4.3.4 in \cite{general-hadamard}.
\end{remark}

\begin{cor}
\label{cor:hadamard-cor1}
\textsl{For a regular positive operator $A$ in $\mathscr{R}$, and a positive operator $B$ in $\mathscr{S}$, the following inequality holds :
\begin{equation}
\label{eqn:hadamard-cor1}
\frac{\Delta(A + B)}{\Delta(A)} \le \frac{\Delta(\Phi(A^{-1})^{-1} + B)}{\Delta(\Phi(A^{-1})^{-1})},
\end{equation}
with equality if and only if $B^{\frac{1}{2}}A^{-1}B^{\frac{1}{2}} \in \mathscr{S}$. In particular, if $B$ is regular, equality holds in (\ref{eqn:hadamard-cor1}) if and only if $A \in \mathscr{S}$.}
\end{cor}
\begin{proof}
 Using the multiplicativity of $\Delta$ and Lemma \ref{lem:sylvester-rule}, we can rewrite both sides of the inequality in the following manner :
\begin{align*}
\frac{\Delta(A + B)}{\Delta(A)} &= \Delta(I + A^{-1}B) = \Delta(I+B^{\frac{1}{2}}A^{-1}B^{\frac{1}{2}}),\\
\frac{\Delta(\Phi(A^{-1})^{-1} + B)}{\Delta(\Phi(A^{-1})^{-1})}&= \Delta(I + \Phi(A^{-1})B) = \Delta (I+B^{\frac{1}{2}}\Phi(A^{-1})B^{\frac{1}{2}}).
\end{align*}
Note that $X := B^{\frac{1}{2}}A^{-1}B^{\frac{1}{2}}$ is a positive operator and thus $I+X$ is a regular positive operator. As $B^{\frac{1}{2}}$ is in $\mathscr{S}$, $\Phi(X) = B^{\frac{1}{2}}\Phi(A^{-1})B^{\frac{1}{2}}$. From Theorem \ref{thm:general-hadamard}, we have that $$\Delta(I+X) \le \Delta( \Phi(I+X)) = \Delta(I + \Phi(X))$$ with equality if and only if $I + X = \Phi(I+X) \Leftrightarrow \Phi(X) = X \Leftrightarrow B^{\frac{1}{2}}A^{-1}B^{\frac{1}{2}} \in \mathscr{S}$. If $B$ is regular, $B^{\frac{1}{2}}A^{-1}B^{\frac{1}{2}} \in \mathscr{S} \Leftrightarrow A^{-1} \in \mathscr{S} \Leftrightarrow A \in \mathscr{S}.$

\end{proof}

\begin{comment}
\begin{proof}
From Theorem \ref{thm:general-hadamard}, we see that $$\Delta(A + B) \le \Delta(\Phi(A+B)) = \Delta(\Phi(A) + \Phi(B)),$$ $$\frac{1}{\Delta(A)} \le \frac{1}{\Delta(\Phi(A^{-1})^{-1})}.$$
Multiplying both the inequalities gives us (\ref{eqn:hadamard-cor1}) with equality if and only if $\Phi(A + B) = A + B$ and $\Phi(A) = A$. Thus equality holds if and only if $\Phi(A) = A$ and $\Phi(B) = B$.

We next prove the inequality when $B$ is not regular. As $A$ is regular, there is an $\varepsilon > 0$ such that $\varepsilon I \le A$. Consider the operators $A_{\varepsilon} := A - \frac{\varepsilon}{2}I$, $B_{\varepsilon} := B + \frac{\varepsilon}{2}I$. We observe that $\frac{\varepsilon}{2}I \le \Phi(A_{\varepsilon}), \frac{\varepsilon}{2}I \le B_{\varepsilon}$ as a result of which $A_{\varepsilon}, B_{\varepsilon}$ are both regular. From what is proved above, we see that 
$$\frac{\Delta(A_{\varepsilon} + B_{\varepsilon})}{\Delta(A_{\varepsilon})} \le \frac{\Delta(\Phi(A_{\varepsilon}^{-1})^{-1} + B_{\varepsilon})}{\Delta(\Phi(A_{\varepsilon}^{-1})^{-1})} $$
Note that the map $A \rightarrow A^{-1}$ is continuous on $GL_1(\mathscr{R})$. Thus taking the limit as $\varepsilon \rightarrow 0^{+}$, we get the desired inequality.
\end{proof}
\end{comment}

\begin{thm}
\label{thm:ineq-square-inverse}
{\textsl For a self-adjoint operator $A$ in $\mathscr{R}$, we have the following inequality :
\begin{equation}
\label{eqn:ineq-square}
\Phi(A)^2 \le \Phi(A^2).
\end{equation}
If $A$ is positive and regular, we have that 
\begin{equation}
\label{eqn:ineq-inverse}
\Phi(A)^{-1} \le \Phi(A^{-1}).
\end{equation}
Further, in inequalities (\ref{eqn:ineq-square}) and (\ref{eqn:ineq-inverse}), equality holds if and only if $\Phi(A) = A$ {\it i.e.}\ $A \in \mathscr{S}$ (for the $A$ under consideration). }
\end{thm}

\begin{proof}
As $\Phi$ is a positive map, note that $\Phi(A)$ is self-adjoint and hence, so is $\Phi(A)-A$. We have $(\Phi(A) - A)^2 \ge 0$ and thus,  $$0 \le \Phi((\Phi(A) - A)^2) = \Phi(\Phi(A)^2 - \Phi(A)A - A\Phi(A) + A^2) = \Phi(A^2) - \Phi(A)^2.$$
Above we used the fact that $\Phi(A\Phi(A)) = \Phi(A) \Phi(A) = \Phi(\Phi(A)A)$. This proves the inequality (\ref{eqn:ineq-square}). In this case, equality holds if and only if $ \Phi((\Phi(A)-A)^2)= 0$. From the faithfulness and positivity of the tracial state, and trace-preserving nature of the conditional expectation, we get that equality holds if and only if $\tau(\Phi(\Phi(A)-A)^2)) = \tau((\Phi(A)-A)^2)) = 0 \Leftrightarrow (\Phi(A)-A)^2 = 0 \Leftrightarrow \Phi(A) = A.$ This completes the proof of inequality (\ref{eqn:ineq-square}).

We next prove inequality (\ref{eqn:ineq-inverse}). From Theorem \ref{thm:cond-exp-comp-positive}, as $\Phi$ is a completely positive map, the map $\Phi \otimes I_2 : \mathscr{R} \otimes M_2(\mathbb{C}) \rightarrow \mathscr{S} \otimes M_2(\mathbb{C})$ is a positive map. Applying $\Phi \otimes I_2$ to the positive operator, 
$$\begin{bmatrix} 
  A  & I\\ 
  I & A^{-1} 
\end{bmatrix} \in \mathscr{R} \otimes M_2(\mathbb{C}),$$ we conclude that $$\begin{bmatrix} 
  \Phi(A) & I\\ 
  I & \Phi(A^{-1}) 
\end{bmatrix} \in \mathscr{S} \otimes M_2(\mathbb{C})$$ is a positive operator. Using Lemma \ref{lem:positive-two-tensor}, we conclude that $\Phi(A)^{-1} \le \Phi(A^{-1})$.

We then investigate conditions under which equality holds. If $\Phi(A) = A$ for a regular positive operator $A$ in $\mathscr{R}$, we have that $A \in \mathscr{S} \Rightarrow A^{-1} \in \mathscr{S} \Rightarrow \Phi(A^{-1}) = A^{-1} = \Phi(A)^{-1}.$ Conversely, if $\Phi(A^{-1})=\Phi(A)^{-1}$, using (\ref{eqn:hadamard}), we see that, $$\frac{1}{\Delta(\Phi(A))} \le \frac{1}{\Delta(A)} = \Delta(A^{-1}) \le \Delta(\Phi(A^{-1})) = \Delta(\Phi(A)^{-1}) = \frac{1}{\Delta(\Phi(A))}.$$
Thus $\Delta(A)^{-1} = \Delta(\Phi(A))^{-1} \Rightarrow \Delta(\Phi(A)) = \Delta(A).$ From the equality condition in Theorem \ref{thm:general-hadamard}, we conclude that $\Phi(A) = A$.
\end{proof}

\begin{cor}
\label{cor:mon-conv}
\textsl{For a positive operator $A$ in $\mathscr{R}$ and $\lambda > 0$, the following inequalities hold:
\begin{itemize}
\item[(i)] $\Phi(A(\lambda I + A)^{-1}) \le \Phi(A)(\lambda I + \Phi(A))^{-1},$
\item[(ii)] $\Phi(A)^2(\lambda I + \Phi(A))^{-1} \le \Phi(A^2(\lambda I+A)^{-1})$
\end{itemize}
In both inequalities, equality holds if and only if $\Phi(A) = A$ {\it i.e.}\ $A \in \mathscr{S}$.}
\end{cor}
\begin{proof}
These inequalities follow from Theorem \ref{thm:ineq-square-inverse}, after noting the following algebraic identitites : $$X(\lambda I + X)^{-1} = I - \lambda I (\lambda I + X)^{-1},$$ $$X^2(\lambda I + X)^{-1} = X - \lambda I + \lambda ^2 I (\lambda I + X)^{-1}.$$
\end{proof}

\begin{thm}
\label{thm:ineq-op-mon}
\textsl{Let $f$ be an operator monotone function defined on the interval $[0, \infty)$. Then for a regular positive operator $A$, we have the following inequality : 
\begin{equation}
\label{eqn:ineq-op-mon}
 \Phi(f(A)) \le f(\Phi(A)),
\end{equation}
with equality if and only if either $f$ is linear with positive slope or $\Phi(A) = A$.}
\end{thm}
\begin{proof}
Let $f$ be operator monotone. From Theorem \ref{thm:int-op-mon}, we have a finite positive measure $\mu$ on $(0, \infty)$ and real numbers $a, b$ with $b \ge 0$, such that $$f(t) = a + bt  + \int_{0}^{\infty} \frac{(\lambda + 1)t}{\lambda + t} \; d\mu (\lambda).$$
Consider the continuous family of operators $\mathbf{H}(\lambda) := \Phi((\lambda + 1) A (\lambda I+A)^{-1}) - (\lambda + 1) \Phi(A)(\lambda I + \Phi(A))^{-1}$ parametrized by $\lambda \in (0, \infty)$. Note that as $A$ is regular, $\mathbf{H}(0)$ is well-defined and equal to $0$. By Corollary \ref{cor:mon-conv} (i), we have that the family $\mathbf{H}$ consists of positive operators and $\mathbf{H}(\lambda) = 0$ for some $\lambda \in (0, \infty)$ if and only if $\Phi(A) = A$. We get the desired inequality below, $$\Phi(f(A)) - f(\Phi(A)) = \int_{0}^{\infty} \mathbf{H}(\lambda) \; d\mu(\lambda) \ge 0.$$

The next step is to find necessary and sufficient conditions for equality in (\ref{eqn:ineq-op-mon}). Note that for any continuous function $f$ on $[0, \infty)$, if $A$ is in $\mathscr{S}$, then $f(A)$ is also in $\mathscr{S}$. Thus $\Phi(A) = A$ implies that $\Phi(f(A)) = f(A) = f(\Phi(A)).$ Also if $f$ is linear, clearly $f(\Phi(A)) = \Phi(f(A))$. 

If $f(\Phi(A)) = \Phi(f(A))$ and  $\Phi(A) \ne A$, from the equality condition in Corollary \ref{cor:mon-conv}, we conclude that $\mathbf{H}(\lambda ') \neq 0$ for any $\lambda ' \in (0, \infty)$. For a vector $x$ in $\mathscr{H}$, one may define a positive continuous function $h_x$ on $(0, \infty)$ by $h_x(\lambda) = \langle \mathbf{H}(\lambda)x, x \rangle $. Note that for each $\lambda '$ in $(0, \infty)$, there is a vector $x_{\lambda '}$ such that $h_{x_{\lambda '}}(\lambda ') > 0$ and thus a neighborhood $N_{\lambda '}$ of $\lambda '$ where $h_{x_{\lambda '}}$ is strictly positive. As $\int_{0}^{\infty}h_x(\lambda) \; d\mu(\lambda) = 0$ for all vectors $x$ in $\mathscr{H}$, in particular, we have that $\int_{0}^{\infty} h_{x_{\lambda '}}(\lambda) \; d\mu(\lambda) = 0$. We conclude that $\mu$ is not supported on $N_{\lambda '}$ for any $\lambda '$ in $(0, \infty)$ and thus $f(t) = a + bt$. Hence if $f(\Phi(A)) = \Phi(f(A))$ and $\Phi(A) \ne A$, we have that $f$ must be a linear function with positive slope.

\end{proof}

\begin{remark}
\label{rmrk:ineq-op-mon-conv}
The above inequality (\ref{eqn:ineq-op-mon}) holds true along with the equality condition, even under the assumption that $f$ is operator monotone on the open interval $(0, \infty)$. Let $A$ be a regular positive operator in $\mathscr{R}$, and $\varepsilon > 0$ be such that $\varepsilon I \le A$. Consider the function $g$ on $[0, \infty)$ defined by $g(x) := f(x + \frac{\varepsilon}{2})$, and the regular operator $A_{\varepsilon} := A - \frac{\varepsilon}{2} I$. Clearly, $g$ is an operator monotone function on $[0, \infty)$. By Theorem \ref{thm:ineq-op-mon}, $g(\Phi(A_{\varepsilon})) \le \Phi(g(A_{\varepsilon}))$ and equality holds if and only if either $g$ is linear with positive slope or $\Phi(A_{\varepsilon}) = A_{\varepsilon}$. Translating in terms of $f$, we get $\Phi(f(A)) = \Phi(f(A_{\varepsilon}+\frac{\varepsilon}{2}I)) \le ( f(\Phi(A_{\varepsilon}) + \frac{\varepsilon}{2}I)  = f(\Phi(A)) $ with equality if and only if either $f$ is linear with positive slope or $\Phi(A) - \frac{\varepsilon}{2}I = A - \frac{\varepsilon}{2}I$ \emph{i.e.} $\Phi(A) = A.$    
\end{remark}

\begin{thm}
\label{thm:ineq-op-conv}
\textsl{Let $f$ be an operator convex function defined on the interval $[0, \infty)$. Then for a regular positive operator $A$ in $\mathscr{R}$, we have the following inequality : 
\begin{equation}
f(\Phi(A)) \le \Phi(f(A)),
\end{equation}
with equality if and only if either $f$ is linear, or $\Phi(A) = A$ {\it i.e.}\ $A \in \mathscr{S}$.}
\end{thm}
\begin{proof}
Let $f$ be operator convex and assume that $f'(0^{+})$ exists. From Theorem \ref{thm:int-op-conv}, we have a finite positive measure $\mu$ on $[0, \infty)$ and real numbers $a, b, c$ with $c$ non-negative, such that $$f(t) = a + bt + ct^2 + \int_{0}^{\infty} \frac{(\lambda + 1) t^2}{\lambda + t} \; d\mu (\lambda).$$
Using Corollary \ref{cor:mon-conv} (ii) and inequality (\ref{eqn:ineq-square}), we may essentially mimic the proof in Theorem \ref{thm:ineq-op-mon} adapting it to the case of operator convex functions. What deserves mention is the disappearance of the quadratic term $cx^2$ in the equality case. We start with the assumption that $\Phi(A) \ne A$. Since $\Phi(f(A)) = f(\Phi(A))$, we must have $c\Phi(A^2) = c \Phi(A)^2$ as $\langle (c\Phi(A^2) - c \Phi(A)^2)x, x \rangle = 0$ for all $x$ in $\mathscr{H}$. Thus, from the equality condition in Theorem \ref{thm:ineq-square-inverse}, we see that $c = 0$. The rest of the proof for the equality case is similar to the case of operator monotone functions. We conclude that if $f(\Phi(A)) = \Phi(f(A))$ and $\Phi(A) \ne A$, $f$ must be a linear function.

Finally we get rid of the assumption of existence of $f'(0^{+})$. As in Remark \ref{rmrk:ineq-op-mon-conv}, let $\varepsilon > 0$ be such that $\varepsilon I \le A$ and define a function $g$ on $[0, \infty)$ by $g(x) := f(x + \frac{\varepsilon}{2})$. Note that $g$ is an operator convex function on $[0, \infty)$ and $g'(0) = f'(\frac{\varepsilon}{2})$ exists. For the regular positve operator $A_{\varepsilon} := A - \frac{\varepsilon}{2}I$, we conclude that $g(\Phi(A_{\varepsilon})) \le \Phi(g(A_{\varepsilon})) \Rightarrow  f(\Phi(A)) = f(\Phi(A_{\varepsilon}) + \frac{\varepsilon}{2}I) \le \Phi(f(A_{\varepsilon}+\frac{\varepsilon}{2}I)) = \Phi(f(A))$ with equality if and only if either $g$ is linear or $\Phi(A_{\varepsilon}) = A_{\varepsilon}$ {\it i.e.}\ $f$ is linear or $\Phi(A) = A$.

\end{proof}

We alert the reader to note that, although the inequalities in Theorem \ref{thm:ineq-op-mon}, Theorem \ref{thm:ineq-op-conv} involve similar-looking quantities, the direction is reversed.

\begin{thm}
\label{cor:ineq-det-op-mon}
\textsl{Let $f : (0, \infty) \rightarrow (0, \infty)$ be a non-constant operator monotone function. Then for a regular positive operator $A$ in $\mathscr{R}$, we have the following inequality, 
\begin{equation}
\label{eqn:ineq-det-op-mon}
\Delta(f(A)) \le \Delta(f(\Phi(A)),
\end{equation} with equality if and only if $\Phi(A) = A$ {\it i.e.}\ $A \in \mathscr{S}$.}
\end{thm}
\begin{proof}
As $\log$ is an operator monotone function on $(0, \infty)$, we observe that $\log f$ is also operator monotone on $(0, \infty)$. As the exponential function is not operator monotone, note that $\log f$ is not linear. Thus from Theorem \ref{thm:ineq-op-mon} and Remark \ref{rmrk:ineq-op-mon-conv}, we have that $$\Phi(\log f(A)) \le \log f(\Phi(A))),$$
with equality if and only if $\Phi(A) = A$. As $\tau$ is a faithful state, using the trace-preserving nature of $\Phi$, we have $$\tau(\Phi(\log f(A))) = \tau(\log f(A)) \le \tau(\log f (\Phi(A))),$$
and equality holds if and only if $\Phi(A) = A.$ Applying the exponential function, which is strictly increasing, to both sides of the inequality, we get the desired inequality with equality if and only if $\Phi(A) = A$.
\end{proof}

\begin{cor}
\label{cor:ineq-det-perturb}
\textsl{For a regular positive operator $A$ in $\mathscr{R}$, we have that 
\begin{equation}
\label{eqn:ineq-det-perturb}
\Delta(I+\Phi(A)^{-1}) \le \Delta(I + A^{-1}),
\end{equation}
with equality if and only if $\Phi(A) = A$ {\it i.e.}\ $A \in \mathscr{S}$.}
\end{cor}
\begin{proof} This follows from Corollary \ref{cor:ineq-det-op-mon}, by choosing $f(x) = (1+\frac{1}{x})^{-1}$ which is a positive-valued operator monotone function on $(0, \infty).$ 
\end{proof}

\begin{thm}
\label{thm:ineq-det-matic}
\textsl{ For a regular positive operator $A$ in $\mathscr{R}$, and a positive operator $B$ in $\mathscr{S}$, the following inequality holds : 
\begin{equation}
\label{eqn:ineq-det-matic}
\frac{\Delta(\Phi(A) + B))}{\Delta(\Phi(A))} \le \frac{\Delta(A + B)}{\Delta(A)}
\end{equation}
If $B$ is regular, equality holds if and only if $\Phi(A) = A$ {\it i.e.}\ $A \in \mathscr{S}$.\\
(Note that when $B=0$, equality holds for any regular positive operator $A$. This illustrates that when $B$ is not regular, the characterizing conditions for equality may not be as simple in form.)}
\end{thm}
\begin{proof}
As $A$ is regular, there is an $\varepsilon > 0$ such that $\varepsilon I \le A$. First we prove the inequality for the regular operators $A_{\varepsilon} := A - \frac{\varepsilon}{2}I$, $B_{\varepsilon} := B + \frac{\varepsilon}{2}I$. We observe that $\frac{\varepsilon}{2}I \le \Phi(A_{\varepsilon}), \frac{\varepsilon}{2}I \le B_{\varepsilon}$ as a result of which $\Phi(A_{\varepsilon}), B_{\varepsilon}$ are also regular. Using the multiplicativity of $\Delta$, we can rewrite both sides of the inequality in the following manner :
\begin{align*}
\frac{\Delta(A_{\varepsilon} + B_{\varepsilon})}{\Delta(A_{\varepsilon})} &= \Delta(I+B_{\varepsilon}^{\frac{1}{2}}A_{\varepsilon}^{-1}B_{\varepsilon}^{\frac{1}{2}}),\\
\frac{\Delta(\Phi(A_{\varepsilon}) + B_{\varepsilon})}{\Delta(\Phi(A_{\varepsilon}))}&=\Delta
(I+B_{\varepsilon}^{\frac{1}{2}}\Phi(A_{\varepsilon})^{-1}B_{\varepsilon}^{\frac{1}{2}}).
\end{align*}
Note that $X := B_{\varepsilon}^{-\frac{1}{2}}A_{\varepsilon}B_{\varepsilon}^{-\frac{1}{2}}$ is a regular positive operator. As $B_{\varepsilon}^{-\frac{1}{2}}$ is in $\mathscr{S}$, $\Phi(X) = B_{\varepsilon}^{-\frac{1}{2}}\Phi(A_{\varepsilon})B_{\varepsilon}^{-\frac{1}{2}}$. From Corollary \ref{cor:ineq-det-op-mon}, $$\Delta(I+\Phi(X)^{-1}) \le \Delta(I+X^{-1}).
$$
which proves the inequality for $A_{\varepsilon}, B_{\varepsilon}$. Note that $\Phi(A) + B = \Phi(A_{\varepsilon}) + B_{\varepsilon}$ and $A + B = A_{\varepsilon} + B_{\varepsilon}$. Thus we have that $$\frac{\Delta(\Phi(A) + B)}{\Delta(\Phi(A_{\varepsilon}))} \le \frac{\Delta(A + B)}{\Delta(A_{\varepsilon})}.$$
Taking the limit as $\varepsilon \rightarrow 0^{+}$, we get the required inequality.

If $B$ is regular, then we may directly follow the above steps without performing the perturbative step of defining $A_{\varepsilon}, B_{\varepsilon}$ and instead defining $X := B^{-\frac{1}{2}}AB^{-\frac{1}{2}}$. Noting that $B^{-\frac{1}{2}}$ is in $\mathscr{S}$ and using the equality condition in Corollary \ref{cor:ineq-det-perturb}, we see that equality holds if and only if $\Phi(X) = X \Leftrightarrow B^{-\frac{1}{2}}\Phi(A)B^{-\frac{1}{2}} = \Phi(B^{-\frac{1}{2}}AB^{-\frac{1}{2}}) = B^{-\frac{1}{2}}AB^{-\frac{1}{2}} \Leftrightarrow \Phi(A) = A$.

\end{proof}

\begin{remark}
\label{rmrk:ineq-det-matic-var}
The following generalized form of inequality (\ref{eqn:ineq-det-matic}), 
\begin{equation}
\label{eqn:ineq-det-matic-var}
\frac{\Delta(\Phi(A) + \Phi(B))}{\Delta(\Phi(A))} = \frac{\Delta(\Phi(A + B))}{\Delta(\Phi(A))} \le \frac{\Delta(A + B)}{\Delta(A)}
\end{equation}
does not hold for an arbitrary choice of positive operators $A, B$ in $\mathscr{R}$ with $A$ being regular. From Theorem 4.1, we have that $\Delta(I + B) \le \Delta (\Phi(I + B)) = \Delta(I + \Phi(B)) $ with equality if and only if $\Phi(I + B) = I + B$ \emph{i.e.} $\Phi(B) = B$. So inequality (\ref{eqn:ineq-det-matic-var}) is clearly untrue if $A = I$ and $B$ is not in $\mathscr{S}$ \emph{i.e.} $\Phi(B) \ne B$. Remark \ref{rmrk:ineq-det-matic-matrices}
\end{remark}

\begin{comment}
Note that choosing $B$ in $\mathscr{S}$ is necessary. The inequality (\ref{eqn:ineq-det-matic}) does not hold for an arbitrary choice of a positive operator $B$ in $\mathscr{R}$. From Theorem \ref{thm:general-hadamard}, we have that $\Delta(I+B) \le \Delta (\Phi(I + B)) = \Delta(I + \Phi(B)) $ with equality if and only if $\Phi(I + B) = I + B$ \emph{i.e.} $\Phi(B) = B$. So inequality (\ref{eqn:ineq-det-matic}) is clearly untrue if $A = I$ and $B$ is not in $\mathscr{S}$ \emph{i.e.} $\Phi(B) \ne B$.
\end{comment}

In the remaining portion of the section, we establish versions of some of the inequalities proved above, in a more general setting. Broadly speaking, conditional expectations on $\mathscr{R}$ are replaced by the more general unital positive maps on $\mathscr{R}$, again denoted by $\Phi$. For a detailed study of positive linear maps between operator algebras, we direct the reader to \cite[St{{\o}}rmer]{stormer-paper} (see also \cite{stormer}). By the Kadison-Schwarz inequality (\cite{kad-schwarz-ineq}), inequality (\ref{eqn:ineq-square}) still holds in this context. In addition, if $\Phi$ were a unital $2$-positive map, then the proof for inequality (\ref{eqn:ineq-inverse}) in Theorem \ref{thm:ineq-square-inverse} goes through. The map $\Phi : \mathscr{R} \rightarrow \mathscr{R}$ is said to be $\tau$-preserving or trace-preserving if $\tau(X) = \tau(\Phi(X))$ for all $X$ in $\mathscr{R}$. A careful scrutiny would reveal that for $\Phi$ a trace-preserving unital $2$-positive map, the determinant inequalities (\ref{eqn:ineq-det-op-mon}), (\ref{eqn:ineq-det-perturb}) are still valid. We mention the appropriate version of the result for unital $2$-positive maps below and leave it to the reader to work out the details. The trade-off for this level of generality is that we are unable to find straightforward conditions for equality.

\begin{thm}
Let $\Phi : \mathscr{R} \rightarrow \mathscr{R}$ be a unital $2$-positive map which is $\tau$-preserving and $A$ be a regular positive operator in $\mathscr{R}$. For a positive-valued operator monotone function $f$ on $(0, \infty)$, we have the following inequality : $$\Delta(f(A)) \le \Delta(f(\Phi(A))).$$
\end{thm}

In addition to the preceding comments, we also explore another direction of generalization below.  We first prove a convexity inequality as a form of Jensen's inequality. Although the basic idea is contained in the proof of \cite[Lemma 3.10]{jensen}, we adapt the relevant parts to our discussion for the sake of clarity and continuity.
\begin{prop}
\label{prop:trace-unit-positive}
\textsl{Let $\Phi : \mathscr{R} \rightarrow \mathscr{R}$ be a unital positive map. Let $f$ be a real-valued continuous convex function defined on the interval $[a, b] \subseteq \mathbb{R}$. Then for every self-adjoint operator $A$ in $\mathscr{R}$ with spectrum in $[a, b]$, $\Phi(A)$ also has spectrum in $[a, b]$ and we have the following inequality : $$\tau(f(\Phi(A))) \le \tau(\Phi(f(A))).$$}
\end{prop}
\begin{proof}
As the spectrum of $A$ is contained in $[a, b]$, we note that $aI \le A \le bI$. Since $\Phi$ is a unital positive map, we conclude that $aI = \Phi(aI)  \le \Phi(A) \le \Phi(bI) = bI$. Thus $\Phi(A)$ also has spectrum in $[a, b]$. 

Let $\mathscr{A}$ be a \emph{masa} of $\mathscr{R}$ containing $\Phi(A).$ By Theorem 2.8, there is a unique trace-preserving conditional expectation $\Psi : \mathscr{R} \rightarrow \mathscr{A}$. Note that $\Psi \circ \Phi$ is a unital positive map into a commutative von Neumann algebra $\mathscr{A}$, and $(\Psi \circ \Phi)(A) = \Phi(A)$. From Theorem \ref{thm:jensen-unital-positive},  we have that $$f(\Phi(A)) = f((\Psi \circ \Phi)(A)) \le (\Psi \circ \Phi)(f(A)).$$
Using the positivity of the trace and trace-preserving nature of $\Psi$, we conclude that $$\tau(f(\Phi(A)) \le \tau(\Psi(\Phi(f(A)))) = \tau(\Phi(f(A))).$$
\end{proof}

\begin{thm}
\label{thm:ineq-det-unit-positive}
\textsl{Let $\Phi : \mathscr{R} \rightarrow \mathscr{R}$ be a trace-preserving unital positive map. Let $f$ be a continuous positive function defined on the interval $[a, b] \subseteq \mathbb{R}$ such that $\log f$ is convex (in other words, $f$ is log-convex). Then for every positive operator $A$ in $\mathscr{R}$, we have the following inequality :
\begin{equation}
\label{eqn:ineq-det-unit-positive}
\Delta(f(\Phi(A)) \le \Delta(f(A))
\end{equation}
}
\end{thm}
\begin{proof}
Using Proposition \ref{prop:trace-unit-positive}, from the convexity of $\log f$ and trace-preserving nature of $\Phi$, we observe that,
 $$\tau(\log f(\Phi(A))) \le \tau(\Phi(\log f(A))) = \tau(\log f(A)).$$
 Thus $\Delta(f(\Phi(A))) = \exp(\tau(\log f(\Phi(A)))) \le \exp(\tau(\log f(A))) = \Delta (f(A)).$
\end{proof}

\begin{cor}
\label{cor:general-hadamard-unit-positive}
\textsl{Let $\Phi : \mathscr{R} \rightarrow \mathscr{R}$ be a trace-preserving unital positive map. Then for every positive operator $A$ in $\mathscr{R}$, we have the following inequality :
\begin{equation}
\Delta(A) \le \Delta(\Phi(A))
\end{equation}
}
Further, if $A$ is also regular, then we have that
\begin{equation}
\Delta(\Phi(A^{-1})^{-1}) \le \Delta(A)
\end{equation}
\end{cor}
\begin{proof}
For $b > \varepsilon > 0$, the function $f : [\varepsilon, b] \rightarrow \mathbb{R}$ defined by $f(x) = \frac{1}{x}$, is a log-convex function. From Theorem \ref{thm:ineq-det-unit-positive}, for a regular positive operator $A$ with spectrum in $[\varepsilon, b]$, we have that $$\Delta(\Phi(A)^{-1}) \le \Delta(A^{-1}).$$
Thus using the multiplicativity of the Fuglede-Kadison determinant, we conclude that $$\Delta(A) \le \Delta(\Phi(A)).$$ 
If $A$ is any positive operator (not necessarily regular), applying the inequality to the regular positive operator $A+\varepsilon I$ (for $\varepsilon > 0$), we note that $\Delta(A+\varepsilon I) \le \Delta(\Phi(A+\varepsilon I)) = \Delta(\Phi(A) + \varepsilon I)$. Keeping in mind that $\Delta$ is a continuous function on $\mathscr{R}$, and taking the limit as $\varepsilon \rightarrow 0^{+}$, we note that $\Delta(A) \le \Delta(\Phi(A))$ for all positive operators $A$ in $\mathscr{R}$.

If $A$ is regular, then by the inequality proved above $\Delta(A^{-1}) \le \Delta(\Phi(A^{-1}))$. Using the multiplicativity of $\Delta$, we conclude that $$\Delta(\Phi(A^{-1})^{-1}) \le \Delta(A)
.$$

\end{proof}

\begin{cor}
\textsl{Let $\Phi : \mathscr{R} \rightarrow \mathscr{R}$ be a trace-preserving unital positive map. Then for every regular positive operator $A$ in $\mathscr{R}$, we have that 
\begin{equation}
\Delta(I+\Phi(A)^{-1}) \le \Delta(I + A^{-1}).
\end{equation} }
\end{cor}
\begin{proof}
As $A$ is a regular positive operator, there is an $\varepsilon > 0$ such that $\varepsilon I \le A$. For $b > \varepsilon > 0$, the function $f : [\varepsilon, b] \rightarrow \mathbb{R}$ defined by $f(x) = \frac{x+1}{x} = 1 + \frac{1}{x}$, is a log-convex function as the second derivative of $\log (\frac{x+1}{x}) = \log (x+1) - \log x$ is $\frac{1}{x^2} - \frac{1}{(x+1)^2}$ which is greater than $0$. The required inequality is just inequality (\ref{eqn:ineq-det-unit-positive}) for the $f$ considered above.
\end{proof}

\begin{cor}
\textsl{Let $\Phi : \mathscr{R} \rightarrow \mathscr{R}$ be a trace-preserving unital positive map. For positive operators $A, B$ in $\mathscr{R}$ with $A$ being regular, the following inequality holds :
\begin{equation}
\label{eqn:ineq-unit-positive-matic}
\frac{\Delta(A + B)}{\Delta(A)} \le \frac{\Delta(\Phi(A) + \Phi(B))}{\Delta(\Phi(A^{-1})^{-1})}
\end{equation} }
\end{cor}
\begin{proof}
Using Corollary \ref{cor:general-hadamard-unit-positive}, we observe that the following inequalities hold,
$$\Delta(A + B) \le \Delta(\Phi(A+B)) = \Delta(\Phi(A) + \Phi(B)),$$ $$\frac{1}{\Delta(A)} = \Delta(A^{-1}) \le \Delta(\Phi(A^{-1})) = \frac{1}{\Delta(\Phi(A^{-1})^{-1})} .$$
Multiplying both the inequalities gives us (\ref{eqn:ineq-unit-positive-matic}).
\end{proof}

\section{Applications}
\label{sec:app}
In the previous section, several determinant inequalities were proved in a general setting. Here we make the appropriate choices of the finite von Neumann algebra $\mathscr{R}$, the von Neumann subalgebra $\mathscr{S}$ of $\mathscr{R}$, and the trace-preserving conditional expectation $\Phi$, to obtain the inequalities mentioned in the introduction. We also provide several subtle improvements to Theorem 1.3 and Theorem 1.4.

1. {\bf Hadamard's Inequality.} We follow the notation in Example \ref{ex:cond-exp-1}. Let $\mathscr{R} = M_n(\mathbb{C})$, $\mathscr{S} = D_n(\mathbb{C})$, and $\Phi : \mathscr{R} \rightarrow \mathscr{S}$ be given by $\Phi(A) := \mathrm{diag}(a_{11}, \cdots, a_{nn})$. For a positive-definite matrix $A$, we have from Theorem \ref{thm:general-hadamard} that $\sqrt[n]{\det A} = \Delta(A) \le \Delta(\Phi(A)) = \sqrt[n]{a_{11} \cdots a_{nn}}.$ Taking $n^{\textrm{th}}$ powers on both sides, we obtain Hadamard's inequality, and equality holds if and only if $\Phi(A) = A$ {\it i.e.}\ $A$ is a diagonal matrix.

For the remaining three inequalities, we are in the setting of Example \ref{ex:cond-exp-2}. Let $\mathscr{R} = M_n(\mathbb{C})$, $\mathscr{S} = M_{n_1}(\mathbb{C})\oplus \cdots \oplus M_{n_k}(\mathbb{C})$, and $\Phi : \mathscr{R} \rightarrow \mathscr{S}$ be given by $\Phi(A) = \mathrm{diag}(A_{11}, \cdots, A_{kk})$.

2. {\bf Fischer's Inequality.}  For a positive-definite matrix $A$, we have from Theorem \ref{thm:general-hadamard} that $\det A = \Delta(A)^n \le \Delta(\Phi(A))^n = \det (\mathrm{diag}(A_{11}, \cdots, A_{kk})) = (\det A_{11}) \cdots (\det A_{kk})$. This gives us Fischer's inequality, and equality holds if and only if $\Phi(A) = A$ {\it i.e.}\ $A$ is a block diagonal matrix.

3. {\bf Proof of Theorem \ref{thm:det-ineq-matic-1} .} Consider a positive-definite matrix $C$ in $M_n(\mathbb{C})$, and a positive-definite matrix $D = \mathrm{diag}(D_1, \cdots, D_k)$ in $M_{n_1}(\mathbb{C})\oplus \cdots \oplus M_{n_k}(\mathbb{C})$. Let the principal diagonal blocks of $C$ be denoted by $C_1, \cdots, C_k$. We observe that,  $$ \frac{\det(C_1 + D_1)}{\det(C_1)} \cdots \frac{\det(C_k + D_k)}{\det(C_k)} = \frac{\det (\mathrm{diag}(C_1 + D_1, \cdots, C_k + D_k))}{\det (\mathrm{diag}(C_1, \cdots, C_k))}= \frac{\det(\Phi(C) + D)}{\det(\Phi(C))}.$$
Thus, from Theorem \ref{thm:ineq-det-matic}, we have that, $$\frac{\det(\Phi(C) + D)}{\det(\Phi(C))} = \left( \frac{\Delta(\Phi(C) + D)}{\Delta(\Phi(C))} \right)^n \le \left( \frac{\Delta(C + D)}{\Delta(C)} \right)^n = \frac{\det(C + D)}{\det(C)}.$$ This proves Theorem \ref{thm:det-ineq-matic-1} and equality holds if and only if $\Phi(C) = C$ {\it i.e.}\ $C$ is in block diagonal form. If $D$ is positive-semidefinite, the inequality still holds but the equality condition may not be applicable as noted in the parenthetical remark following the statement of Theorem \ref{thm:ineq-det-matic}.

4. {\bf Proof of Theorem \ref{thm:det-ineq-matic-2} .} Consider a positive-definite matrix $C$ in $M_n(\mathbb{C})$, and a positive-semidefinite matrix $D = \mathrm{diag}(D_1, \cdots, D_k)$ in $M_{n_1}(\mathbb{C})\oplus \cdots \oplus M_{n_k}(\mathbb{C})$. Let $C_1, \cdots, C_k$ be positive-definite matrices such that the principal diagonal blocks of $C^{-1}$ are given by $C_1 ^{-1}, \cdots, C_k ^{-1}$; in other words, $\Phi(C^{-1}) = $ $\mathrm{diag}(C_1 ^{-1}, \cdots, C_k ^{-1})$ or $\Phi(C^{-1})^{-1} = $ $\mathrm{diag}(C_1, \cdots, C_k)$. We observe that, $$\frac{\det(C_1 + D_1)}{\det(C_1)} \cdots \frac{\det(C_k + D_k)}{\det(C_k)} = \frac{\det (\mathrm{diag}(C_1 + D_1, \cdots, C_k + D_k))}{\det (\mathrm{diag}(C_1, \cdots, C_k))}= \frac{\det(\Phi(C^{-1})^{-1} + D)}{\det(\Phi(C^{-1})^{-1})}.$$ Thus, from Corollary \ref{cor:hadamard-cor1}, we have that,
$$\frac{\det(\Phi(C^{-1})^{-1} + D)}{\det(\Phi(C^{-1})^{-1})} = \left( \frac{\Delta(\Phi(C^{-1})^{-1} + D)}{\Delta(\Phi(C^{-1})^{-1})} \right)^n \ge \left( \frac{\Delta(C + D)}{\Delta(C)} \right)^n = \frac{\det(C + D)}{\det(C)}.$$ This proves Theorem \ref{thm:det-ineq-matic-2} and equality holds if and only if $D^{\frac{1}{2}} C^{-1}D^{\frac{1}{2}}$ is in block diagonal form {\it i.e.}\ $D^{\frac{1}{2}} C^{-1}D^{\frac{1}{2}}$ is in $M_{n_1}(\mathbb{C})\oplus \cdots \oplus M_{n_k}(\mathbb{C})$.

\begin{remark}
\label{rmrk:ineq-det-matic-matrices}
The necessity of using $\textrm{diag}(D_1, \cdots, D_k)$ in Theorem \ref{thm:det-ineq-matic-2} instead of an arbitrary positive-definite matrix $D$ with $D_1, \cdots, D_k$ in its principal diagonal blocks, is illuminated by remark \ref{rmrk:ineq-det-matic-var}. For a positive definite matrix $C$ with principal diagonal blocks $C_1, \cdots, C_k$ and a positive semi-definite matrix $D$ with principal diagonal blocks $D_1, \cdots, D_k$, the inequality $$ \frac{\det(C_1 + D_1)}{\det(C_1)} \cdots \frac{\det(C_k + D_k)}{\det(C_k)} \le \frac{\det(C + D)}{\det(C)}$$ does not hold true in general.
\end{remark}

\end{document}